\documentclass[12pt]{amsart}
\usepackage{times}
\usepackage{amsfonts}
\usepackage{amssymb}
\usepackage{bbm}
\usepackage{times}
\usepackage{amssymb}
\usepackage{amscd}
\usepackage{graphicx}

\usepackage{amsmath}
\usepackage{amssymb}

\usepackage{amsmath}
\usepackage{amsfonts}
\usepackage{amscd}
\usepackage{latexsym}
\usepackage{amsthm}
\usepackage{amssymb}
\usepackage{amsmath}

\theoremstyle{plain}
\newtheorem{theorem}{Theorem}[section]

\newtheorem{lemma}[theorem]{Lemma}

\newtheorem{definition}[theorem]{Definition}

\begin{document}

\vskip 0.5cm

\title[asymptotically tracially  approximation ${\rm C^*}$-algebras]{{\bf Divisible properties  for  asymptotically tracially  approximation ${\rm C^*}$-algebras}}
\author{Qingzhai Fan and  Jiahui Wang}
\address{ Qingzhai Fan\\ Department of Mathematics\\ Shanghai Maritime University\\
Shanghai\\China
\\  201306 }
\email{qzfan@shmtu.edu.cn}

\address{Jiahui  Wang\\ Department of Mathematics\\ Shanghai Maritime University\\
Shanghai\\China
\\  201306 }
\email{641183019@qq.com}



\begin{abstract} We show that  the following divisible  properties  of the
  ${\rm C^*}$-algebras in a class $\mathcal{P}$ are inherited by simple unital ${\rm C^*}$-algebras in the
 class of asymptotically tracially in $\mathcal{P}$: $(1)$ $m$-almost divisible, $(2)$ weakly $(m, n)$-divisible.
 \end{abstract}
\maketitle

\section{Introduction}

 The Elliott program for the classification of amenable
 ${\rm C^*}$-algebras might be said to have begun with the ${\rm K}$-theoretical
 classification of AF algebras in \cite{E1}. Since then, many classes of
 ${\rm C^*}$-algebras have been classified by the Elliott invariant.
  Lin introduced the concepts of tracial topological rank (\cite{L0} and \cite{L1}).
  Instead of assuming inductive limit structure, he started with a
  certain abstract tracial approximation property.
  This abstract tracial approximation property  has proved to be very important in the classification of simple amenable ${\rm C^*}$-algebras. For example,
   it led to the classification of unital simple separable amenable ${\rm C^*}$-algebras with finite nuclear dimension in the UCT class (see \cite{G4}, \cite{G5} \cite{EZ5}, \cite{TWW1}).

   In the quest to classify simple separable nuclear  ${\rm C^*}$-algebras, as suggested by G. A. Elliott, it has become necessary to invoke some regularity
 property of the  ${\rm C^*}$-algebras. There are three regularity properties of particular interest: tensorial absorption of the Jiang-Su algebra $\mathcal{Z}$, also called $\mathcal{Z}$-stability;  finite nuclear dimension; and strict comparison of positive elements.
 The latter can be reformulated as an algebraic property of Cuntz semigroup, called almost unperforated. Winter and Toms have conjectured that these three fundamental properties are equivalent for all separable, simple, nuclear  ${\rm C^*}$-algebras.

Hirshberg and Oroviz introduced tracially $\mathcal{Z}$-stable in \cite{HO}, they  showed that $\mathcal{Z}$-stable is equivalent to tracially $\mathcal{Z}$-stable for unital simple separable amenable
  ${\rm C^*}$-algebra in \cite{HO}.

Inspired by  Hirshberg and Oroviz's tracially $\mathcal{Z}$-stable, and  inspired by the work of Elliott,  Gong,  Lin, and  Niu in \cite {EGLN2}, \cite{GL2} and \cite{GL3}, and also in order to search a tracial version of Toms-Winter conjecture, Fu and Lin  introduced  a class of tracial nuclear dimensional ${\rm C^*}$-algebras in \cite{FL}. Also in \cite{FL}, Fu and Lin introduced  a class of asymptotically tracially  ${\rm C^*}$-algebra. They show that a unital separable simple  ${\rm C^*}$-algebra $A$ has tracial nuclear dimensional no more than $k$  is equivalent to the statement that A is asymptotically
tracially in ${{\mathcal{F}_n}}$, where $\mathcal{F}_n$ be the class of all $\rm {C^*}$-algebras with nuclear dimension at most $n$.

In \cite{FL}, Fu and Lin show that  the following  properties  of unital
${\rm C^*}$-algebras in a class $\mathcal{P}$ are inherited by simple unital ${\rm C^*}$-algebras of asymptotically
tracially in $\mathcal{P}$:

$(1)$ stably finite,

$(2)$ quasidiagonal ${\rm C^*}$-algebras,

$(3)$ purely infinite simple ${\rm C^*}$-algebras,

$(4)$ the Cuntz semigroup is almost unperforated, and

$(5)$ strict comparison property.

In \cite{FF}, Fan and Fang show that certain comparison  properties  of unital
${\rm C^*}$-algebras in a class $\mathcal{P}$ are inherited by simple unital ${\rm C^*}$-algebras of asymptotically
tracially in $\mathcal{P}$.

The property  of $m$-almost divisible  was  introduced by Robert and Tikuisis in \cite{RT}.
The property  of weakly $(m, n)$-divisible was introduced by Robert and  R{\o}rdam in \cite{KM}.

In this paper, we show that  the following divisible   properties  of unital
${\rm C^*}$-algebras in a class $\mathcal{P}$ are inherited by simple unital ${\rm C^*}$-algebras in  the class of asymptotically tracially in $\mathcal{P}$:

$(1)$ $m$-almost divisible,

$(2)$ weakly $(m, n)$-divisible.

\section{Definitions and preliminaries}

Let $A$ be a  ${\rm C^*}$-algebra. For two positive elements $a,~b$ in $A$
we say that $a$ is Cuntz subequivalent to $b$ and write $a\precsim b$ if there are
$(r_n)_{n=1}^\infty \subseteq A$ such that $\lim_{n\to \infty}\|r_nbr_n^*-a\|=0.$
We say that $a$ and $b$ are Cuntz equivalent (written $a\sim b$) if $a\precsim b$ and $b\precsim a$. We write $\langle a\rangle$ for the equivalence class of $a$.
(Cuntz equivalent was first introduced Cuntz in   \cite{CJ}).

For a ${\rm C^*}$-algebra $A$, define ${\rm M}_{\infty}(A)=\bigcup_{n\in {\mathbb{N}}}{\rm M}_n(A)$, and for $a\in {\rm M}_n(A)_+$ and $b\in{\rm M}_m(A)_+$, define $a\oplus b:=diag(a, b)\in {\rm M}_{n+m}(A)_+$.
Suppose $a, b\in {\rm M}_{\infty}(A)_+$. Then $a\in {\rm M}_n(A)_+$ and $b\in{\rm M}_m(A)_+$ for some integer $n,m$. We say that $a$ is Cuntz subequivalent to $b$ and write $a\precsim b$ if $a\oplus 0_{\max({(m-n)},0)}\precsim b\oplus 0_{\max({(n-m)},0)}$ as elements in $ {\rm M}_{\max{(n,m)}}(A)_+$.

Let $A$ be a  unital ${\rm C^*}$-algebra. Recall that a positive element $a\in A$ is called purely positive if $a$ is
not Cuntz equivalent to a projection.

Given $a$ in ${\rm M}_{\infty}(A)_+$ and $\varepsilon>0,$ we denote by $(a-\varepsilon)_+$ the element of ${\rm C^*}(a)$ corresponding (via the functional calculus) to the function $f(t)={\max (0, t-\varepsilon)},  ~ t\in \sigma(a)$. By the functional calculus, it follows in a straightforward manner that $((a-\varepsilon_1)_+-\varepsilon_2)_+=(a-(\varepsilon_1+\varepsilon_2))_+.$

The property  of $m$-almost divisible  was  introduced by Robert and Tikuisis in \cite{RT}.

\begin{definition}(\cite{RT}.)\label{def:2.4} Let $m\in \mathbb{N}$. We say that $A$ is  $m$-almost divisible if for each $a\in {\rm M}_{\infty}(A)_+,$ $k\in \mathbb{N}$ and $\varepsilon>0,$ there exists $ b\in{\rm M}_{\infty}(A)_+$ such that $k\langle b\rangle\leq \langle a\rangle$ and $\langle (a-\varepsilon)_+\rangle\leq (k+1)(m+1)\langle b\rangle$.
\end{definition}

The property  of weakly $(m, n)$-divisible was introduced by  Robert and  R{\o}rdam in \cite{KM}.

\begin{definition}(\cite{KM}.)\label{def:2.4}  Let $A$ be unital ${\rm C^*}$-algebra. Let $m, n\geq 1$ be integers.
 We say $A$ is weakly $(m, n)$-divisible,  if for every $u$ in ${\rm W}(A)$, any $\varepsilon>0$,  there exist elements $x_1,x_2,\cdots, x_n\in {\rm W}(A)$, such that $mx_j\leq u$ for all $j=1,2,\cdots, n$ and $(u-\varepsilon)_+\leq x_1+x_2+\cdots+ x_n$.
\end{definition}

\begin{definition} Let $A$ and $B$ be  ${\rm C^*}$-algebra, let $\varphi:A\to B$
be a map, let $\mathcal{G}\subset A$, and let $\varepsilon>0$. The map $\varphi$ is called $\mathcal{G}$-$\varepsilon$-multiplicative, or called $\varepsilon$-multiplicative on $\mathcal{G}$, if
for any $x,y\in F$, $\|\varphi(xy)-\varphi(x)\varphi(y)\|<\varepsilon$. If, in addition, for any
$x\in \mathcal{G}$, $|\|\varphi(x)\|-\|x\||<\varepsilon$, then we say $\varphi$ is an $\mathcal{G}$-$\varepsilon$-approximate embedding.
\end{definition}

Fu and Lin  introduced a class of asymptotically tracially  approximation ${\rm C^*}$-algebras in \cite{FL}.
 The following  definition is  Definition 3.1 in \cite{FL}.

\begin{definition}(\cite{FL}.)\label{def:1.5} Let $A$ be a unital ${\rm C^*}$-algebra. Let $\mathcal{P}$ be  a class of ${\rm C^*}$-algebra. We say $A$ is  asymptotically tracially
in $\mathcal{P}$,  if for any
 $\varepsilon>0,$ any finite
subset $F\subseteq A,$ and any  non-zero element $a\geq 0,$ there exist
a ${\rm C^*}$-algebra $B$  in $\mathcal{P}$ and completely positive  contractive linear maps  $\alpha:A\to B$ and  $\beta_n: B\to A$, and $\gamma_n:A\to A$ such that

$(1)$ $\|x-\gamma_n(x)-\beta_n(\alpha(x))\|<\varepsilon$ for all $x\in F$ and for all $n\in {\mathbb{N}}$,

$(2)$ $\alpha$ is an $F$-$\varepsilon$ approximate embedding,

$(3)$ $\lim_{n\to \infty}\|\beta_n(xy)-\beta_n(x)\beta_n(y)\|=0$ and $\lim_{n\to \infty}\|\beta_n(x)\|=\|x\|$ for all $x,y\in B$, and

$(4)$ $\gamma_n(1_A)\precsim a$ for all $n\in \mathbb{N}$.
\end{definition}

The following theorem is Proposition 3.8 in \cite{FL}.
\begin{theorem}(\cite{FL}.)  Let $\mathcal{P}$ be a class of ${\rm C^*}$-algebras. Let $A$ be a simple unital
${\rm C^*}$-algebra which is asymptotically tracially in  $\mathcal{P}$. Then the following conditions holds: For any
 $\varepsilon>0,$ any finite
subset $F\subseteq A,$ and any  non-zero element $a\geq 0,$ there exist
a ${\rm C^*}$-algebra $B$  in $\mathcal{P}$ and completely positive  contractive linear maps  $\alpha:A\to B$ and  $\beta_n: B\to A$, and $\gamma_n:A\to A\cap\beta_n(B)^{\perp}$ such that

$(1)$ the map $\alpha$ is unital  completely positive   linear map,  $\beta_n(1_B)$ and $\gamma_n(1_A)$ are projections and  $\beta_n(1_B)+\gamma_n(1_A)=1_A$ for all $n\in \mathbb{N}$,

$(2)$ $\|x-\gamma_n(x)-\beta_n(\alpha(x))\|<\varepsilon$ for all $x\in F$ and for all $n\in {\mathbb{N}}$,

$(3)$ $\alpha$ is an $F$-$\varepsilon$-approximate embedding,

$(4)$ $\lim_{n\to \infty}\|\beta_n(xy)-\beta_n(x)\beta_n(y)\|=0$ and $\lim_{n\to \infty}\|\beta_n(xy)\|=\|x\|$ for all $x,y\in B$, and

$(5)$ $\gamma_n(1_A)\precsim a$ for all $n\in \mathbb{N}$.
\end{theorem}

\begin{lemma}(\cite{FL}.)\label{lem:2.6} If the class $\mathcal{P}$ is closed under tensoring with matrix algebras and
 under passing to  unital
hereditary ${\rm C^*}$-subalgebras, then  the class which is asymptotically tracially in $\mathcal{P}$  is closed under
tensoring with matrix algebras  and under passing to unital hereditary ${\rm C^*}$-subalgebras.
\end{lemma}

The following lemma is obvious,  and we omit the proof.
\begin{lemma}\label{lem:2.7} The $m$-almost divisible (or weakly $(m, n)$-divisible) is preserved  under  tensoring with
matrix algebras and  under passing to  unital hereditary ${\rm C^*}$-subalgebras.
\end{lemma}

\section{The main results}
\begin{theorem}\label{thm:3.1}
Let $\mathcal{P}$ be a class of unital
${\rm C^*}$-algebras which are $m$-almost divisible.  Let $A$ be a unital separable simple ${\rm C^*}$-algebra. If  $A$ is asymptotically tracially in $\mathcal{P}$, then $A$ is  $m$-almost divisible.
\end{theorem}
\begin{proof} We need to show that there is $b\in {\rm M}_{\infty}(A)_+$  such that  $kb\precsim a$  and $(a-\varepsilon)_+\precsim (k+1)(m+1)b$
 for any $a\in A_+,$  $\varepsilon>0$ and $k\in \mathbb{N}.$
  We may assume that $\|a\|=1.$

With  $F=\{a\},$   any $\varepsilon'>0$ with $\varepsilon'$ sufficiently small,  since $A$ is asymptotically tracially in $\mathcal{P}$, there exist
a ${\rm C^*}$-algebra $B$  in $\mathcal{P}$ and completely positive  contractive linear maps  $\alpha:A\to B$ and  $\beta_n: B\to A$, and $\gamma_n:A\to A\cap\beta_n(B)^{\perp}$ such that

$(1)$ the map $\alpha$ is unital  completely positive   linear map, $\beta_n(1_B)$ and $\gamma_n(1_A)$ are all projections, and  $\beta_n(1_B)+\gamma_n(1_A)=1_A$ for all $n\in \mathbb{N}$,

$(2)$ $\|x-\gamma_n(x)-\beta_n(\alpha(x))\|<\varepsilon'$ for all $x\in F$ and for all $n\in {\mathbb{N}}$,

$(3)$ $\alpha$ is an $F$-$\varepsilon'$-approximate embedding,

$(4)$ $\lim_{n\to \infty}\|\beta_n(xy)-\beta_n(x)\beta_n(y)\|=0$ and $\lim_{n\to \infty}\|\beta_n(x)\|=\|x\|$ for all $x,y\in B$.

 Since  $B$  is   $m$-almost divisible, and $\alpha(a)\in B$,  there exists  $b_1\in B$ such that $kb_1\precsim \alpha(a)$
 and $(\alpha(a)-\varepsilon')_+\precsim (k+1)(m+1)b_1$.

 We  divide the proof into two cases.

\textbf{Case (1)}, we assume that $(\alpha(a)-\varepsilon')_+$ is Cuntz equivalent to a projection.

\textbf{(1.1)}, If  $(\alpha(a)-\varepsilon')_+$  is not Cuntz equivalent to $(k+1)(m+1)b_1$.

By Theorem 2.1 (2) in \cite{EFF}, we may assume that there exist non-zero $c\in B_+$ such that $(\alpha(a)-\varepsilon')_++ c\precsim (k+1)(m+1)b_1$.

Since $kb_1\precsim \alpha(a)$, there exist $v\in M_k(B)_+$ such that
$$\|v^*diag(\alpha(a), 0\otimes 1_{k-1})v-b_1\otimes 1_k\|<\bar{\varepsilon}.$$
We assume that $\|v\|\leq M(\bar{\varepsilon})$, by $(4)$, there exists sufficiently large integer $N_1$ such that for any $n>N_1$, we have  $$\|\beta_n\otimes id_{M_{k}}(v^*)diag(\beta_n\alpha(a), 0\otimes 1_{k-1})\beta_n\otimes id_{M_{k}}(v)-\beta_n(b_1)\otimes 1_k\|<\varepsilon'.$$
Therefore we have $$k(\beta_n(b_1)-4\varepsilon')_+\precsim (\beta_n\alpha(a)-2\varepsilon')_+.$$

Since $(\alpha(a)-\varepsilon')_++ c\precsim (k+1)(m+1)b_1$, there exist $w\in M_{(k+1)(m+1)}(B)_+$ such that
$$\|w^*(b_1\otimes 1_{(k+1))m+1)}w-diag((\alpha(a)-\varepsilon')_++ c, 0\otimes 1_{(k+1)(m+1)-1)})\|<\bar{\varepsilon}.$$
We assume that $\|w\|\leq N(\bar{\varepsilon})$, by $(4)$, there exists sufficiently large integer $N_2$ such that for any $N_2<n$ we have  $$\|\beta_n\otimes id_{M_{(k+1)(m+1)}}(w^*)\beta_n(b_1)\otimes 1_{(k+1)(m+1)}\beta_n\otimes id_{M_{(k+1)(m+1)}}(w)$$$$-diag(\beta_n\alpha((a)-\varepsilon')_++ c, 0\otimes 1_{(k+1)(m+1)-1})\|<\varepsilon'.$$
 Therefore we have $$(\beta_n\alpha(a)-6\varepsilon')_++c\precsim (\beta_n(b_1)-2\varepsilon')_+.$$

For sufficiently large $n>max\{N_1,N_2\}$, with  $F=\{\gamma_n(a)\},$   any $\varepsilon''>0$ with $\varepsilon''$ sufficiently small,  let $E=\gamma_n(1_A)A\gamma_n(1_A)$, by Lemma 2.7, $E$ is  asymptotically tracially in $\mathcal{P}$, there exist
a ${\rm C^*}$-algebra $D$  in $\mathcal{P}$ and completely positive  contractive linear maps  $\alpha':E\to D$ and  $\beta_n': D\to E$, and $\gamma_n':E\to E\cap\beta_n'(D)^{\perp}$ such that

$(1)'$ the map $\alpha'$ is unital  completely positive   linear map, $\beta_n'(1_D)$ and $\gamma_n'(1_E)$ are all projections, $\beta_n'(1_D)+\gamma_n'(1_E)=1_E$ for all $n\in \mathbb{N}$,

$(2)'$ $\|x-\gamma_n'(x)-\beta_n'(\alpha'(x))\|<\varepsilon''$ for all $x\in G$ and for all $n\in {\mathbb{N}}$,

$(3)'$ $\alpha'$ is an $G$-$\varepsilon''$-approximate embedding,

$(4)'$ $\lim_{n\to \infty}\|\beta_n'(xy)-\beta_n'(x)\beta_n(y)\|=0$ and $\lim_{n\to \infty}\|\beta_n'(x)\|=\|x\|$ for all $x,y\in D$, and

$(5)'$ $\gamma_n'\gamma(1_E)\precsim \beta_n(c)$ for all $n\in \mathbb{N}$.

   Since  $D$ is $m$-almost divisible, and $(\alpha'\gamma_n(a)-3\varepsilon')_+\in D$,  there exists  $b_2\in D_+$ such that $kb_2\precsim (\alpha'\gamma_n(a)-3\varepsilon')_+$
 and $(\alpha'\gamma_n(a)-4\varepsilon')_+\precsim (k+1)(m+1)b_2$.

 With the same argument as above we have
  $$k(\beta_n'(b_2)-2\varepsilon')_+\precsim (\beta_n'\alpha'\gamma_n(a)-\varepsilon')_+$$ and
  $$(\beta_n'\alpha'\gamma_n(a)-6\varepsilon')_+\precsim (\beta_n'(b_2)-2\varepsilon')_+.$$

 Therefore we have
\begin{eqnarray}
\label{Eq:eq1}
  &&k((\beta_n(b_1)-2\varepsilon')_+\oplus (\beta_n'(b_2)-2\varepsilon')_+)
  \thicksim k((\beta_n(b_1)-2\varepsilon')_++ (\beta_n'(b_2)-2\varepsilon')_+)\nonumber\\
  &&\precsim (\beta_n\alpha(a)-\varepsilon')_+\oplus(\beta_n'\alpha'\gamma_n(a)-\varepsilon')_+ \nonumber\\
   &&\precsim (\beta_n\alpha(a)-\varepsilon')_+\oplus(\beta_n'\alpha'\gamma_n(a)-\varepsilon')_+
   +(\gamma_n'\gamma_n(a)-3\varepsilon)_+ \nonumber\\
  &&\precsim a,\nonumber
\end{eqnarray}
and we also have
\begin{eqnarray}
\label{Eq:eq1}
  &&(a-\varepsilon)_+\nonumber\\
  &&\precsim  (\beta_n\alpha(a)-4\varepsilon')_+\oplus (\beta_n'\alpha'\gamma_n(a)-6\varepsilon')_+ \oplus (\gamma_n'\gamma_n(a)-4\varepsilon')_+\nonumber\\
 &&\precsim  (\beta_n\alpha(a)-4\varepsilon')_+\oplus (\beta_n'\alpha'\gamma_n(a)-4\varepsilon')_+ \oplus \gamma_n'\gamma_n(1_E)\nonumber\\
&&\precsim (\beta_n\alpha(a)-4\varepsilon')_+\oplus (\beta_n'\alpha'\gamma_n(a)-4\varepsilon)_+ \oplus \beta_n(c)\nonumber\\
&&\precsim (k+1)(m+1)(\beta_n(b_1)-2\varepsilon')_+\oplus (k+1)(m+1)(\beta_n(b_1)-2\varepsilon')_+.\nonumber\\
&&\precsim (k+1)(m+1)(\beta_n(b_1)-2\varepsilon')_+\oplus (k+1)(m+1)(\beta_n(b_1)-2\varepsilon')_+.\nonumber
\end{eqnarray}

\textbf{(1.2)},
If $(\alpha(a)-\varepsilon')_+$   is  Cuntz equivalent to $(k+1)(m+1)b_1$.

Since $k(b_1\oplus b_1)\precsim \alpha(a)$, there exist $v\in M_k(B)_+$ such that
$$\|v^*diag(\alpha(a), 0\otimes 1_{2k-1})v-b_1\otimes 1_{2k}\|<\bar{\varepsilon}.$$
We assume that $\|v\|\leq M(\bar{\varepsilon})$, by $(4)$, there exists sufficiently large integer $N_1$  such that for any $n>N_1$, we have  $$\|\beta_n\otimes id_{M_{2k}}(v^*)diag(\beta_n\alpha(a), 0\otimes 1_{k-1}\beta_n\otimes id_{M_{2k}}(v)-\beta_n(b_1)\otimes 1_{2k}\|<\varepsilon'.$$
Therefore we have $$k(\beta_n(b_1)\oplus\beta_n(b_1)-4\varepsilon')_+\precsim (\beta_n\alpha(a)-2\varepsilon')_+.$$

With the same argument, we have
 $$(\beta_n\alpha(a)-6\varepsilon')_++c\precsim (\beta_n(b_1)-2\varepsilon')_+.$$

For sufficiently large $n>max\{N_1,N_2\}$, with  $F=\{\gamma_n(a)\},$   any $\varepsilon''>0$ with $\varepsilon''$ sufficiently small,  let $E=\gamma_n(1_A)A\gamma_n(1_A)$, by Lemma 2.7, $E$ is  asymptotically tracially in $\mathcal{P}$, there exist
a ${\rm C^*}$-algebra $D$  in $\mathcal{P}$ and completely positive  contractive linear maps  $\alpha':E\to D$ and  $\beta_n': D\to E$, and $\gamma_n':E\to E\cap\beta_n'(D)^{\perp}$ such that

$(1)'$ the map $\alpha'$ is unital  completely positive   linear map, $\beta_n'(1_D)$ and $\gamma_n'(1_E)$ are all projections, $\beta_n'(1_D)+\gamma_n'(1_E)=1_E$ for all $n\in \mathbb{N}$,

$(2)'$ $\|x-\gamma_n'(x)-\beta_n'(\alpha'(x))\|<\varepsilon''$ for all $x\in G$ and for all $n\in {\mathbb{N}}$,

$(3)'$ $\alpha'$ is an $G$-$\varepsilon''$-approximate embedding,

$(4)'$ $\lim_{n\to \infty}\|\beta_n'(xy)-\beta_n'(x)\beta_n(y)\|=0$ and $\lim_{n\to \infty}\|\beta_n'(x)\|=\|x\|$ for all $x,y\in D$, and

$(5)'$ $\gamma_n'\gamma(1_A)\precsim \beta_n(b_1)$ for all $n\in \mathbb{N}$.

 Since  $D$ is   $m$-almost divisible, and $(\beta_n'\alpha'\gamma_n(a)-\varepsilon')_+\in B$,  there exists  $b_2\in D_+$ such that $kb_2\precsim (\beta_n'\alpha'\gamma_n(a)-\varepsilon')_+$
 and $(\beta_n'\alpha'\gamma_n(a)-2\varepsilon')_+\precsim (k+1)(m+1)b_2$.

  With the same argument as above we have
  $$k(\beta_n'(b_2)-2\varepsilon')_+\precsim (\beta_n'\alpha'\gamma_n(a)-\varepsilon')_+$$ and
  $$(\beta_n'\alpha'\gamma_n(a)-6\varepsilon')_+\precsim (\beta_n'((b_1-2\varepsilon')_+).$$

  Therefore we have
\begin{eqnarray}
\label{Eq:eq1}
  &&k((\beta_n(b_1\oplus b_1)-2\varepsilon')_+\oplus (\beta_n'(b_2)-2\varepsilon')_+)
  \thicksim k((\beta_n(b_1)-2\varepsilon')_++ (\beta_n'(b_2)-2\varepsilon')_+)\nonumber\\
  &&\precsim (\beta_n\alpha(a)-\varepsilon')_+\oplus(\beta_n'\alpha'\gamma_n(a)-\varepsilon')_+ \nonumber\\
   &&\precsim (\beta_n\alpha(a)-\varepsilon')_+\oplus(\beta_n'\alpha'\gamma_n(a)-\varepsilon')_+
   +(\gamma_n'\gamma_n(a)-3\varepsilon)_+ \nonumber\\
  &&\precsim a,\nonumber
\end{eqnarray}
and we also have
\begin{eqnarray}
\label{Eq:eq1}
  &&(a-\varepsilon)_+\nonumber\\
  &&\precsim  (\beta_n\alpha(a)-4\varepsilon')_+\oplus (\beta_n'\alpha'\gamma_n(a)-6\varepsilon')_+ \oplus (\gamma_n'\gamma_n(a)-4\varepsilon')_+\nonumber\\
 &&\precsim  (\beta_n\alpha(a)-4\varepsilon')_+\oplus (\beta_n'\alpha'\gamma_n(a)-4\varepsilon')_+ \oplus \gamma_n'\gamma_n(1_E)\nonumber\\
&&\precsim (\beta_n\alpha(a)-4\varepsilon')_+\oplus (\beta_n'\alpha'\gamma_n(a)-4\varepsilon)_+ \oplus \beta_n(c)\nonumber\\
&&\precsim (k+1)(m+1)(\beta_n(b_1)-2\varepsilon')_+\oplus (k+1)(m+1)(\beta_n(b_1)-2\varepsilon')_+.\nonumber\\
&&\precsim (k+1)(m+1)(\beta_n(b_1)-2\varepsilon')_+\oplus (k+1)(m+1)(\beta_n(b_1\oplus b_1)-2\varepsilon')_+.\nonumber
\end{eqnarray}

 \textbf{Case (2)}, we assume that $(\alpha(a)-\varepsilon')_+$ is not Cuntz equivalent to a projection.

 By Theorem 2.1 (4) in \cite{EFF}, there is a non-zero positive element $d$  such that
 $(\alpha(a)-2\varepsilon')_++d\precsim (\alpha(a)-\varepsilon')_+$.

Since $(\alpha(a)-\varepsilon')_++ c\precsim (k+1)(m+1)b_1$, there exist $w\in M_{(k+1)(m+1)}(B)_+$ such that
$$\|w^*b_1\otimes 1_{(k+1)(m+1)}w-diag(\alpha(a)-\varepsilon')_++ c, 0\otimes 1_{(k+1)(m+1)-1}\|<\bar{\varepsilon}.$$
We assume that $\|w\|\leq N(\bar{\varepsilon})$, by $(4)$, there exists sufficiently large integer $n$ such that $$\|\beta_n(w^*)\beta_n(b_1)\otimes 1_{(k+1)(m+1)}\beta_n(w)-diag(\beta_n\alpha((a)-\varepsilon')_++ c), 0\otimes 1_{(k+1)(m+1)}\|<\bar{\varepsilon}.$$
 Therefore we have $$(\beta_n\alpha(a)-6\varepsilon')_++c)\precsim (\beta_n((b_1-2\varepsilon')_+).$$

Since $kb_1\precsim \alpha(a)$, there exist $v\in M_k(B)_+$ such that
$$\|v^*diag(\alpha(a), 0\otimes 1_{k-1})v-b_1\otimes 1_k\|<\bar{\varepsilon}.$$
We assume that $\|v\|\leq M(\bar{\varepsilon})$, by $(4)$, there exists sufficiently large integer $N_1$ such that for any $n>N_1$, we have  $$\|\beta_n\otimes id_{M_{k}}(v^*)diag(\beta_n\alpha(a), 0\otimes 1_{k-1})\beta_n\otimes id_{M_{k}}(v)-\beta_n(b_1)\otimes 1_k\|<\varepsilon'.$$
Therefore we have $$k(\beta_n(b_1)-4\varepsilon')_+\precsim (\beta_n\alpha(a)-2\varepsilon')_+.$$

For sufficiently large $n>max\{N_1,N_2\}$, with  $F=\{\gamma_n(a)\},$   any $\varepsilon''>0$ with $\varepsilon''$ sufficiently small,  let $E=\gamma_n(1_A)A\gamma_n(1_A)$, by Lemma 2.7, $E$ is  asymptotically tracially in $\mathcal{P}$, there exist
a ${\rm C^*}$-algebra $D$  in $\mathcal{P}$ and completely positive  contractive linear maps  $\alpha':E\to D$ and  $\beta_n': D\to E$, and $\gamma_n':E\to E\cap\beta_n'(D)^{\perp}$ such that

$(1)'$ the map $\alpha'$ is unital  completely positive   linear map, $\beta_n'(1_D)$ and $\gamma_n'(1_E)$ are all projections, $\beta_n'(1_D)+\gamma_n'(1_E)=1_E$ for all $n\in \mathbb{N}$,

$(2)'$ $\|x-\gamma_n'(x)-\beta_n'(\alpha'(x))\|<\varepsilon''$ for all $x\in G$ and for all $n\in {\mathbb{N}}$,

$(3)'$ $\alpha'$ is an $G$-$\varepsilon''$-approximate embedding,

$(4)'$ $\lim_{n\to \infty}\|\beta_n'(xy)-\beta_n'(x)\beta_n(y)\|=0$ and $\lim_{n\to \infty}\|\beta_n'(x)\|=\|x\|$ for all $x,y\in D$, and

$(5)'$ $\gamma_n'\gamma(1_A)\precsim \beta_n(b_1)$ for all $n\in \mathbb{N}$.

 Since  $D$ is   $m$-almost divisible, and $(\beta_n'\alpha'\gamma_n(a)-\varepsilon')_+\in B$,  there exists  $b_2\in D_+$ such that $kb_2\precsim (\beta_n'\alpha'\gamma_n(a)-\varepsilon')_+$
 and $(\beta_n'\alpha'\gamma_n(a)-2\varepsilon')_+\precsim (k+1)(m+1)b_2$.

  With the same argument as above we have
  $$k(\beta_n'(b_2)-2\varepsilon')_+\precsim (\beta_n'\alpha'\gamma_n(a)-\varepsilon')_+$$ and
  $$(\beta_n'\alpha'\gamma_n(a)-6\varepsilon')_+\precsim (\beta_n'((b_1-2\varepsilon')_+).$$

  Therefore we have
\begin{eqnarray}
\label{Eq:eq1}
  &&k((\beta_n(b_1\oplus b_1)-2\varepsilon')_+\oplus (\beta_n'(b_2)-2\varepsilon')_+)
  \thicksim k((\beta_n(b_1)-2\varepsilon')_++ (\beta_n'(b_2)-2\varepsilon')_+)\nonumber\\
  &&\precsim (\beta_n\alpha(a)-\varepsilon')_+\oplus(\beta_n'\alpha'\gamma_n(a)-\varepsilon')_+ \nonumber\\
   &&\precsim (\beta_n\alpha(a)-\varepsilon')_+\oplus(\beta_n'\alpha'\gamma_n(a)-\varepsilon')_+
   +(\gamma_n'\gamma_n(a)-3\varepsilon)_+ \nonumber\\
  &&\precsim a,\nonumber
\end{eqnarray}
and we also have
\begin{eqnarray}
\label{Eq:eq1}
  &&(a-\varepsilon)_+\nonumber\\
  &&\precsim  (\beta_n\alpha(a)-4\varepsilon')_+\oplus (\beta_n'\alpha'\gamma_n(a)-6\varepsilon')_+ \oplus (\gamma_n'\gamma_n(a)-4\varepsilon')_+\nonumber\\
 &&\precsim  (\beta_n\alpha(a)-4\varepsilon')_+\oplus (\beta_n'\alpha'\gamma_n(a)-4\varepsilon')_+ \oplus \gamma_n'\gamma_n(1_E)\nonumber\\
&&\precsim (\beta_n\alpha(a)-4\varepsilon')_+\oplus (\beta_n'\alpha'\gamma_n(a)-4\varepsilon)_+ \oplus \beta_n(c)\nonumber\\
&&\precsim (k+1)(m+1)(\beta_n(b_1)-2\varepsilon')_+\oplus (k+1)(m+1)(\beta_n(b_1)-2\varepsilon')_+.\nonumber\\
&&\precsim (k+1)(m+1)(\beta_n(b_1)-2\varepsilon')_+\oplus (k+1)(m+1)(\beta_n(b_1\oplus b_1)-2\varepsilon')_+.\nonumber
\end{eqnarray}
\end{proof}

\begin{theorem}\label{thm:3.2}
 Let $\mathcal{P}$ be a class of  unital
${\rm C^*}$-algebras which are   weakly $(m, n)$-divisible.  Let $A$ be a unital separable simple ${\rm C^*}$-algebra. If  $A$ is asymptotically tracially in $\mathcal{P}$, then $A$  is   weakly $(m, n)$-divisible.
\end{theorem}
\begin{proof}
We need to show that  for any $a\in {\rm M}_{\infty}(A)_+$, any $\varepsilon>0$,  there exist $x_1,x_2, \cdots$, $x_n \in {\rm M}_{\infty}(A)_+$  such that $x_j\oplus x_j\oplus \cdots \oplus x_j\precsim a$ for all $1\leq j\leq n$, where $x_j$ repeat $m$ times,  and $(a-\varepsilon)_+\precsim \oplus_{i=1}^{n}x_i$.

We may assume $a\in A_+$ and $\|a\|\leq 1$.

With  $F=\{a\},$   any $\varepsilon'>0$ with $\varepsilon'$ sufficiently small,  since $A$ is asymptotically tracially in $\mathcal{P}$, there exist
a ${\rm C^*}$-algebra $B$  in $\mathcal{P}$ and completely positive  contractive linear maps  $\alpha:A\to B$ and  $\beta_n: B\to A$, and $\gamma_n:A\to A\cap\beta_n(B)^{\perp}$ such that

$(1)$ the map $\alpha$ is unital  completely positive   linear map, $\beta_n(1_B)$ and $\gamma_n(1_A)$ are all projections, and  $\beta_n(1_B)+\gamma_n(1_A)=1_A$ for all $n\in \mathbb{N}$,

$(2)$ $\|x-\gamma_n(x)-\beta_n(\alpha(x))\|<\varepsilon'$ for all $x\in F$ and for all $n\in {\mathbb{N}}$,

$(3)$ $\alpha$ is an $F$-$\varepsilon'$-approximate embedding,

$(4)$ $\lim_{n\to \infty}\|\beta_n(xy)-\beta_n(x)\beta_n(y)\|=0$ and $\lim_{n\to \infty}\|\beta_n(x)\|=\|x\|$ for all $x,y\in B$.

  Since $B$ has weakly $(m,n)$-divisible, there exist
 $x_1',x_2', \cdots,x_n' \in {\rm M}_{\infty}(B)_+$ such that $x_j'\oplus x_j'\oplus \cdots \oplus x_j'\precsim \alpha(a)$ where $x_j'$ repeat $m$ times  and $ (\alpha(a)-\varepsilon')_+ \precsim \oplus_{i=1}^{n}x_i'$.

We  divide the proof into two cases.

\textbf{Case (1)},  we assume that $(\alpha(a)-\varepsilon')_+$ is   Cuntz equivalent to a projection.

\textbf{(1.1)}, If  $(\alpha(a)-\varepsilon')_+$  is not Cuntz equivalent to
$\oplus_{i=1}^{n}x_i'$.

By Theorem 2.1 (2) in \cite{EFF}, we may assume that there exist non-zero $c\in B_+$ such that $(\alpha(a)-\varepsilon')_++ c\precsim \oplus_{i=1}^{n}x_i'$.

Since $x_j'\oplus x_j'\oplus \cdots \oplus x_j'\precsim \alpha(a)$, there exist $v\in M_k(B)_+$ such that
$$\|v^*diag(\alpha(a), 0\otimes 1_{m-1})v-x_j'\otimes 1_m\|<\bar{\varepsilon}.$$
We assume that $\|v\|\leq M(\bar{\varepsilon})$, by $(4)$, there exists sufficiently large integer $N_1$ such that for any $n>N_1$, we have  $$\|\beta_n\otimes id_{M_{m}}(v^*)diag(\beta_n\alpha(a), 0\otimes 1_{k-1})\beta_n\otimes id_{M_{k}}(v)-\beta_n(x_j')\otimes 1_m\|<\varepsilon'.$$
Therefore we have $$k(\beta_n(x_j)\oplus\beta_n(x_j')\oplus\cdots\oplus\beta_n(x_j')-4\varepsilon')_+\precsim (\beta_n\alpha(a)-2\varepsilon')_+.$$

Since  $(\alpha(a)-\varepsilon')_++ c\precsim \oplus_{i=1}^{n}x_i'$
 there exist $w\in M_{n}(B)_+$ such that
$$\|w^*(\oplus_{i=1}^{n}x_i')w-diag((\alpha(a)-\varepsilon')_++ c, 0\otimes 1_{n-1})\|<\bar{\varepsilon}.$$
We assume that $\|w\|\leq N(\bar{\varepsilon})$, by $(4)$, there exists sufficiently large integer $N_2$ such that for any $N_2<n$ we have  $$\|\beta_n\otimes id_{M_{n}}(w^*)\beta_n(\oplus_{i=1}^{n}x_i')\otimes 1_{n}\beta_n\otimes id_{M_{n}}(w)$$$$-diag(\beta_n\alpha((a)-\varepsilon')_++ c, 0\otimes 1_{n})\|<\varepsilon'.$$
 Therefore we have $$(\beta_n\alpha(a)-6\varepsilon')_++c\precsim (\beta_n(\oplus_{i=1}^{n}x_i')-2\varepsilon')_+.$$

For sufficiently large $n>max\{N_1,N_2\}$, with  $F=\{\gamma_n(a)\},$   any $\varepsilon''>0$ with $\varepsilon''$ sufficiently small,  let $E=\gamma_n(1_A)A\gamma_n(1_A)$, by Lemma 2.7, $E$ is  asymptotically tracially in $\mathcal{P}$, there exist
a ${\rm C^*}$-algebra $D$  in $\mathcal{P}$ and completely positive  contractive linear maps  $\alpha':E\to D$ and  $\beta_n': D\to E$, and $\gamma_n':E\to E\cap\beta_n'(D)^{\perp}$ such that

$(1)'$ the map $\alpha'$ is unital  completely positive   linear map, $\beta_n'(1_D)$ and $\gamma_n'(1_E)$ are all projections, $\beta_n'(1_D)+\gamma_n'(1_E)=1_E$ for all $n\in \mathbb{N}$,

$(2)'$ $\|x-\gamma_n'(x)-\beta_n'(\alpha'(x))\|<\varepsilon''$ for all $x\in G$ and for all $n\in {\mathbb{N}}$,

$(3)'$ $\alpha'$ is an $G$-$\varepsilon''$-approximate embedding,

$(4)'$ $\lim_{n\to \infty}\|\beta_n'(xy)-\beta_n'(x)\beta_n(y)\|=0$ and $\lim_{n\to \infty}\|\beta_n'(x)\|=\|x\|$ for all $x,y\in D$, and

$(5)'$ $\gamma_n'\gamma(1_E)\precsim \beta_n(c)$ for all $n\in \mathbb{N}$.

   Since  $D$ is   $m$-almost divisible, and $(\alpha'\gamma_n(a)-3\varepsilon')_+\in D$,  there exists  $b_2\in D_+$ such that $kb_2\precsim (\alpha'\gamma_n(a)-3\varepsilon')_+$
 and $(\alpha'\gamma_n(a)-4\varepsilon')_+\precsim (k+1)(m+1)b_2$.

 With the same argument as above we have
  $$k(\beta_n'(\oplus_{i=1}^{n}x_i'')-2\varepsilon')_+\precsim (\beta_n'\alpha'\gamma_n(a)-\varepsilon')_+$$ and
  $$(\beta_n'\alpha'\gamma_n(a)-6\varepsilon')_+\precsim (\beta_n'(\oplus_{i=1}^{n}x_i'')-2\varepsilon')_+.$$

 Therefore we have
\begin{eqnarray}
\label{Eq:eq1}
  &&k((\beta_n(\oplus_{i=1}^{n}x_i')-2\varepsilon')_+\oplus (\beta_n'(\oplus_{i=1}^{n}x_i'')-2\varepsilon')_+)
  \thicksim k((\beta_n(b_1)-2\varepsilon')_++ (\beta_n'(b_2)-2\varepsilon')_+)\nonumber\\
  &&\precsim (\beta_n\alpha(a)-\varepsilon')_+\oplus(\beta_n'\alpha'\gamma_n(a)-\varepsilon')_+ \nonumber\\
   &&\precsim (\beta_n\alpha(a)-\varepsilon')_+\oplus(\beta_n'\alpha'\gamma_n(a)-\varepsilon')_+
   +(\gamma_n'\gamma_n(a)-3\varepsilon)_+ \nonumber\\
  &&\precsim a,\nonumber
\end{eqnarray}
and we also have
\begin{eqnarray}
\label{Eq:eq1}
  &&(a-\varepsilon)_+\nonumber\\
  &&\precsim  (\beta_n\alpha(a)-4\varepsilon')_+\oplus (\beta_n'\alpha'\gamma_n(a)-6\varepsilon')_+ \oplus (\gamma_n'\gamma_n(a)-4\varepsilon')_+\nonumber\\
 &&\precsim  (\beta_n\alpha(a)-4\varepsilon')_+\oplus (\beta_n'\alpha'\gamma_n(a)-4\varepsilon')_+ \oplus \gamma_n'\gamma_n(1_E)\nonumber\\
&&\precsim (\beta_n\alpha(a)-4\varepsilon')_+\oplus (\beta_n'\alpha'\gamma_n(a)-4\varepsilon)_+ \oplus \beta_n(c)\nonumber\\
&&\precsim (k+1)(m+1)(\beta_n(b_1)-2\varepsilon')_+\oplus (k+1)(m+1)(\beta_n(b_1)-2\varepsilon')_+.\nonumber\\
&&\precsim (k+1)(m+1)(\beta_n(\oplus_{i=1}^{n}x_i')-2\varepsilon')_+\oplus (k+1)(m+1)(\beta_n(\oplus_{i=1}^{n}x_i')-2\varepsilon')_+.\nonumber
\end{eqnarray}

\textbf{(1.1)}, If  $(\alpha(a)-\varepsilon')_+$  is not Cuntz equivalent to
$\oplus_{i=1}^{n}x_i'$.

With  $F=\{\gamma_n(a)\},$   any $\varepsilon''>0$ with $\varepsilon'$ sufficiently small,  let $E=\gamma_n(1_A)A\gamma_n(1_A)$, by Lemma 2.7, $E$ is  asymptotically tracially in $\mathcal{P}$, there exist
a ${\rm C^*}$-algebra $D$  in $\mathcal{P}$ and completely positive  contractive linear maps  $\alpha':E\to D$ and  $\beta_n': D\to E$, and $\gamma_n':E\to E\cap\beta_n'(D)^{\perp}$ such that

$(1)'$ the map $\alpha'$ is unital  completely positive   linear map, $\beta_n'(1_D)$ and $\gamma_n'(1_E)$ are all projections, $\beta_n'(1_D)+\gamma_n'(1_E)=1_E$ for all $n\in \mathbb{N}$,

$(2)'$ $\|x-\gamma_n'(x)-\beta_n'(\alpha'(x))\|<\varepsilon''$ for all $x\in G$ and for all $n\in {\mathbb{N}}$,

$(3)'$ $\alpha'$ is an $G$-$\varepsilon''$-approximate embedding,

$(4)'$ $\lim_{n\to \infty}\|\beta_n'(xy)-\beta_n'(x)\beta_n(y)\|=0$ and $\lim_{n\to \infty}\|\beta_n'(x)\|=\|x\|$ for all $x,y\in D$, and

$(5)'$ $\gamma_n'(\gamma(1_A))\precsim \beta_n\alpha(c)$ for all $n\in \mathbb{N}$.

Since $D$ is  weakly $(m,n)$-divisible,
 there exist $x_1'',x_2'', \cdots,x_n'' \in {\rm M}_{\infty}(D)_+$ such that
$x_j''\oplus x_j''\oplus \cdots \oplus x_j''\precsim  (\beta_n'\alpha'\gamma_n(a)-2\varepsilon)_+$ where $x_j''$ repeat $m$ times  and $(\beta_n'\alpha'\gamma_n(a)-3\varepsilon)_+\precsim \oplus_{i=1}^{n}x_i''$.

 With the same argument as above we have
  $$k(\beta_n'(b_2)-2\varepsilon')_+\precsim (\beta_n'\alpha'\gamma_n(a)-\varepsilon')_+$$ and
  $$(\beta_n'\alpha'\gamma_n(a)-6\varepsilon')_+\precsim (\beta_n'((b_1-2\varepsilon')_+).$$

 Therefore we have
\begin{eqnarray}
\label{Eq:eq1}
  &&((x_j'\oplus r)\oplus x_j''))\oplus((x_j'\oplus r)\oplus x_j''))\oplus\cdots \oplus((x_j'\oplus r)\oplus x_j''))\nonumber\\
  &&\precsim  (\beta_n\alpha(a)-2\varepsilon)_+\oplus  (\beta_n'\alpha'\gamma_n(a)-2\varepsilon)_+\nonumber\\
&&\precsim a,\nonumber
\end{eqnarray}
and
\begin{eqnarray}
\label{Eq:eq1}
  &&(x_i'\oplus x_i'')\oplus(x_i'\oplus x_i'')\oplus\cdots \oplus(x_i'\oplus x_i'')\nonumber\\
 &&\precsim  (\beta_n\alpha(a)-2\varepsilon)_+\oplus  (\beta_n'\alpha'\gamma_n(a)-2\varepsilon)_+\nonumber\\
&&\precsim a,\nonumber
\end{eqnarray}
for all $i\neq j$ and $1\leq i\leq n$ where $(x_i'\oplus x_i'')$ repeat $m$ times.

 We also have
\begin{eqnarray}
\label{Eq:eq1}
  &&(a-20\varepsilon)_+\nonumber\\
  &&\precsim (\beta_n\alpha(a)-3\varepsilon)_+\oplus (\beta_n'\alpha'\gamma_n(a)-3\varepsilon)_+\oplus (\gamma_n'\gamma_n(a)-4\varepsilon)_+\nonumber\\
  &&\precsim (\beta_n\alpha(a)-3\varepsilon)_+\oplus (\beta_n'\alpha'\gamma_n(a)-3\varepsilon)_+\oplus (\gamma_n'\gamma_n(1_A)-\varepsilon)_+\nonumber\\
    &&\precsim (a'-3\varepsilon)_+\oplus (\beta_n'\alpha'\gamma_n(a)-3\varepsilon)_+\oplus r\nonumber\\
 &&\precsim \oplus_{i=1,i\neq j}^{n}(x_i'\oplus x_i'')\oplus ((x_j'\oplus r)\oplus x_j'').\nonumber
\end{eqnarray}

\textbf{(1.1.1.2)}, If  $x_1',x_2', \cdots,x_k' \in {\rm M}_{\infty}(B)_+$ are all projections,
and $(a'-3\varepsilon)_+< \oplus_{i=1}^{k}x_i'$.
 Then there exists a nonzero projection $s$ such that $(a'-3\varepsilon)_+\oplus s \precsim \oplus_{i=1}^{k}x_i'$.

Since $ (\alpha(a)-3\varepsilon)_+ \precsim \oplus_{i=1}^{n}x_i'$, .
Since $kb_1\precsim \alpha(a)$, there exist $v\in M_k(B)_+$ such that
$$\|v^*diag(\alpha(a), 0\otimes 1_{k-1}v-b_1\otimes 1\|<\bar{\varepsilon}.$$
We assume that $\|v\|\leq M(\bar{\varepsilon})$, by $(4)$, there exists sufficiently large integer $n$ such that $$\|\beta_n(v^*)diag(\beta_n\alpha(a), 0\otimes 1_{k-1}\beta_n(v)-\beta_n(b_1)\otimes 1\|<\varepsilon'.$$
Therefore we have $$k(\beta_n(b_1)-2\varepsilon')_+\precsim (\beta_n\alpha(a)-\varepsilon')_+.$$

With  $F=\{\gamma_n(a)\},$   any $\varepsilon'>0$ with $\varepsilon'$ sufficiently small,  let $E=\gamma_n(1_A)A\gamma_n(1_A)$, by Lemma 2.7, $E$ is  asymptotically tracially in $\mathcal{P}$, there exist
a ${\rm C^*}$-algebra $D$  in $\mathcal{P}$ and completely positive  contractive linear maps  $\alpha':E\to D$ and  $\beta_n': D\to E$, and $\gamma_n':E\to E\cap\beta_n'(D)^{\perp}$ such that

$(1)'$ the map $\alpha'$ is unital  completely positive   linear map, $\beta_n'(1_D)$ and $\gamma_n'(1_E)$ are all projections, $\beta_n'(1_D)+\gamma_n'(1_E)=1_E$ for all $n\in \mathbb{N}$,

$(2)'$ $\|x-\gamma_n'(x)-\beta_n'(\alpha'(x))\|<\varepsilon''$ for all $x\in G$ and for all $n\in {\mathbb{N}}$,

$(3)'$ $\alpha'$ is an $G$-$\varepsilon''$-approximate embedding,

$(4)'$ $\lim_{n\to \infty}\|\beta_n'(xy)-\beta_n'(x)\beta_n(y)\|=0$ and $\lim_{n\to \infty}\|\beta_n'(x)\|=\|x\|$ for all $x,y\in D$, and

$(5)'$ $\gamma_n'(\gamma(1_A))\precsim \beta_n\alpha(c)$ for all $n\in \mathbb{N}$.

Since $D$ is  weakly $(m,n)$-divisible,
 there exist $x_1'',x_2'', \cdots,x_n'' \in {\rm M}_{\infty}(D)_+$ such that
$x_j''\oplus x_j''\oplus \cdots \oplus x_j''\precsim  (\beta_n'\alpha'\gamma_n(a)-2\varepsilon)_+$ where $x_j''$ repeat $m$ times  and $(\beta_n'\alpha'\gamma_n(a)-3\varepsilon)_+\precsim \oplus_{i=1}^{n}x_i''$.
With the same argument as above we have
  $$k(\beta_n'(b_2)-2\varepsilon')_+\precsim (\beta_n'\alpha'\gamma_n(a)-\varepsilon')_+$$ and
  $$(\beta_n'\alpha'\gamma_n(a)-6\varepsilon')_+\precsim (\beta_n'((b_1-2\varepsilon')_+).$$

 Therefore we have
\begin{eqnarray}
\label{Eq:eq1}
  &&(x_i'\oplus x_i'')\oplus(x_i'\oplus x_i'')\oplus\cdots \oplus(x_i'\oplus x_i'')\nonumber\\
 &&\precsim  (a'-2\varepsilon)_+\oplus  (\beta_n'\alpha'\gamma_n(a)-2\varepsilon)_+\nonumber\\
&&\precsim a\oplus a,\nonumber
\end{eqnarray}
for  $1\leq i\leq n$ where $(x_i'\oplus x_i'')$ repeat $m$ times.

We also have
\begin{eqnarray}
\label{Eq:eq1}
  &&(a-20\varepsilon)_+\nonumber\\
  &&\precsim (\beta_n\alpha(a)-3\varepsilon)_+\oplus (\beta_n'\alpha'\gamma_n(a)-3\varepsilon)_+\oplus (\gamma_n'\gamma_n(a)-2\varepsilon)_+\nonumber\\
  &&\precsim (\beta_n\alpha(a)-3\varepsilon)_+\oplus (\beta_n'\alpha'\gamma_n(a)-3\varepsilon)_+\oplus (\gamma_n'\gamma_n(1_A)-\varepsilon)_+\nonumber\\
   &&\precsim (\beta_n\alpha(a)-3\varepsilon)_+\oplus (\beta_n'\alpha'\gamma_n(a)-4\varepsilon)_+\oplus s\nonumber\\
 &&\precsim \oplus_{i=1}^{n}(x_i'\oplus x_i'').\nonumber
\end{eqnarray}

\textbf{(1.1.1.3)}, we  assume that there is a purely positive element $x_1'$. Since $(a'-2\varepsilon)_+\precsim \oplus_{i=1}^{n}x_i'$, for any $\varepsilon>0$, there exists $\delta>0$, such that
 $(a'-4\varepsilon)_+\precsim (x_1'-\delta)_+ \oplus_{i=2}^{n}x_i'$,

By Theorem 2.1 (4) in \cite{EFF},  there exists a  nonzero positive element $d$  such that
 $(x_1'-\delta)_++d\precsim x_1'$.

Since $ (\alpha(a)-3\varepsilon)_+ \precsim \oplus_{i=1}^{n}x_i'$, .
Since $kb_1\precsim \alpha(a)$, there exist $v\in M_k(B)_+$ such that
$$\|v^*diag(\alpha(a), 0\otimes 1_{k-1}v-b_1\otimes 1\|<\bar{\varepsilon}.$$
We assume that $\|v\|\leq M(\bar{\varepsilon})$, by $(4)$, there exists sufficiently large integer $n$ such that $$\|\beta_n(v^*)diag(\beta_n\alpha(a), 0\otimes 1_{k-1}\beta_n(v)-\beta_n(b_1)\otimes 1\|<\varepsilon'.$$
Therefore we have $$k(\beta_n(b_1)-2\varepsilon')_+\precsim (\beta_n\alpha(a)-\varepsilon')_+.$$

With  $F=\{\gamma_n(a)\},$   any $\varepsilon'>0$ with $\varepsilon'$ sufficiently small,  let $E=\gamma_n(1_A)A\gamma_n(1_A)$, by Lemma 2.7, $E$ is  asymptotically tracially in $\mathcal{P}$, there exist
a ${\rm C^*}$-algebra $D$  in $\mathcal{P}$ and completely positive  contractive linear maps  $\alpha':E\to D$ and  $\beta_n': D\to E$, and $\gamma_n':E\to E\cap\beta_n'(D)^{\perp}$ such that

$(1)'$ the map $\alpha'$ is unital  completely positive   linear map, $\beta_n'(1_D)$ and $\gamma_n'(1_E)$ are all projections, $\beta_n'(1_D)+\gamma_n'(1_E)=1_E$ for all $n\in \mathbb{N}$,

$(2)'$ $\|x-\gamma_n'(x)-\beta_n'(\alpha'(x))\|<\varepsilon''$ for all $x\in G$ and for all $n\in {\mathbb{N}}$,

$(3)'$ $\alpha'$ is an $G$-$\varepsilon''$-approximate embedding,

$(4)'$ $\lim_{n\to \infty}\|\beta_n'(xy)-\beta_n'(x)\beta_n(y)\|=0$ and $\lim_{n\to \infty}\|\beta_n'(x)\|=\|x\|$ for all $x,y\in D$, and

$(5)'$ $\gamma_n'(\gamma(1_A))\precsim \beta_n\alpha(c)$ for all $n\in \mathbb{N}$.

Since $D$ is weakly $(m,n)$-divisible,
 there exist $x_1'',x_2'', \cdots,x_n'' \in {\rm M}_{\infty}(D)_+$ such that
$x_j''\oplus x_j''\oplus \cdots \oplus x_j''\precsim  (\gamma_n'\gamma_n(a)-2\varepsilon)_+$ where $x_j''$ repeat $m$ times  and $(\gamma_n'\gamma_n(a)-3\varepsilon)_+\precsim \oplus_{i=1}^{n}x_i''$.
With the same argument as above we have
  $$k(\beta_n'(b_2)-2\varepsilon')_+\precsim (\beta_n'\alpha'\gamma_n(a)-\varepsilon')_+$$ and
  $$(\beta_n'\alpha'\gamma_n(a)-6\varepsilon')_+\precsim (\beta_n'((b_1-2\varepsilon')_+).$$

 Therefore we have
\begin{eqnarray}
\label{Eq:eq1}
  &&(x_i'\oplus x_i'')\oplus(x_i'\oplus x_i'')\oplus\cdots \oplus(x_i'\oplus x_i'')\nonumber\\
 &&\precsim  (\beta_n\alpha(a)-2\varepsilon)_+\oplus  (\beta_n'\alpha'\gamma_n(a)-2\varepsilon)_+\nonumber\\
&&\precsim a,\nonumber
\end{eqnarray}
for  $1\leq i\leq n$ where $(x_i'\oplus x_i'')$ repeat $m$ times.

We also have
\begin{eqnarray}
\label{Eq:eq1}
  &&(a-20\varepsilon)_+\nonumber\\
  &&\precsim (\beta_n\alpha(a)-4\varepsilon)_+\oplus (\beta_n'\alpha'\gamma_n(a)-3\varepsilon)_+\oplus (\gamma_n'\gamma_n(a)-2\varepsilon)_+\nonumber\\
  &&\precsim (\beta_n\alpha(a)-4\varepsilon)_+\oplus (\beta_n'\alpha'\gamma_n(a)-3\varepsilon)_+\oplus (\gamma_n'\gamma_n(1_A)-\varepsilon)_+\nonumber\\
    &&\precsim (\beta_n\alpha(a)-4\varepsilon)_+\oplus (\beta_n'\alpha'\gamma_n(a)-4\varepsilon)_+\oplus d\nonumber\\
 &&\precsim \oplus_{i=1}^{n}(x_i'\oplus x_i'').\nonumber
\end{eqnarray}

\textbf{(1.1.2)}  We assume that there exists nonzero projection $r$ such that  $(a'-3\varepsilon)_++r\precsim (a'-2\varepsilon)_+$.

Since $ (\alpha(a)-3\varepsilon)_+ \precsim \oplus_{i=1}^{n}x_i'$, .
Since $kb_1\precsim \alpha(a)$, there exist $v\in M_k(B)_+$ such that
$$\|v^*diag(\alpha(a), 0\otimes 1_{k-1}v-b_1\otimes 1\|<\bar{\varepsilon}.$$
We assume that $\|v\|\leq M(\bar{\varepsilon})$, by $(4)$, there exists sufficiently large integer $n$ such that $$\|\beta_n(v^*)diag(\beta_n\alpha(a), 0\otimes 1_{k-1}\beta_n(v)-\beta_n(b_1)\otimes 1\|<\varepsilon'.$$
Therefore we have $$k(\beta_n(b_1)-2\varepsilon')_+\precsim (\beta_n\alpha(a)-\varepsilon')_+.$$

With  $F=\{\gamma_n(a)\},$   any $\varepsilon'>0$ with $\varepsilon'$ sufficiently small,  let $E=\gamma_n(1_A)A\gamma_n(1_A)$, by Lemma 2.7, $E$ is  asymptotically tracially in $\mathcal{P}$, there exist
a ${\rm C^*}$-algebra $D$  in $\mathcal{P}$ and completely positive  contractive linear maps  $\alpha':E\to D$ and  $\beta_n': D\to E$, and $\gamma_n':E\to E\cap\beta_n'(D)^{\perp}$ such that

$(1)'$ the map $\alpha'$ is unital  completely positive   linear map, $\beta_n'(1_D)$ and $\gamma_n'(1_E)$ are all projections, $\beta_n'(1_D)+\gamma_n'(1_E)=1_E$ for all $n\in \mathbb{N}$,

$(2)'$ $\|x-\gamma_n'(x)-\beta_n'(\alpha'(x))\|<\varepsilon''$ for all $x\in G$ and for all $n\in {\mathbb{N}}$,

$(3)'$ $\alpha'$ is an $G$-$\varepsilon''$-approximate embedding,

$(4)'$ $\lim_{n\to \infty}\|\beta_n'(xy)-\beta_n'(x)\beta_n(y)\|=0$ and $\lim_{n\to \infty}\|\beta_n'(x)\|=\|x\|$ for all $x,y\in D$, and

$(5)'$ $\gamma_n'(\gamma(1_A))\precsim \beta_n\alpha(c)$ for all $n\in \mathbb{N}$.

Since $D$ is weakly $(m,n)$-divisible,
 there exist $x_1'',x_2'', \cdots,x_n'' \in {\rm M}_{\infty}(D)_+$ such that
$x_j''\oplus x_j''\oplus \cdots \oplus x_j''\precsim  (\beta_n'\alpha'\gamma_n(a)-2\varepsilon)_+$ where $x_j''$ repeat $m$ times  and $(\beta_n'\alpha'\gamma_n(a)-3\varepsilon)_+\precsim \oplus_{i=1}^{n}x_i''$.
With the same argument as above we have
  $$k(\beta_n'(b_2)-2\varepsilon')_+\precsim (\beta_n'\alpha'\gamma_n(a)-\varepsilon')_+$$ and
  $$(\beta_n'\alpha'\gamma_n(a)-6\varepsilon')_+\precsim (\beta_n'((b_1-2\varepsilon')_+).$$

 Therefore we have
\begin{eqnarray}
\label{Eq:eq1}
  &&(y_i'\oplus x_i'')\oplus(y_i'\oplus x_i'')\oplus\cdots \oplus(y_i'\oplus x_i'')\nonumber\\
 &&\precsim  (\beta_n\alpha(a)-\varepsilon)_+\oplus  (\beta_n'\alpha'\gamma_n(a)-2\varepsilon)_+\nonumber\\
&&\precsim a,\nonumber
\end{eqnarray}
for  $1\leq i\leq n$ where $(x_i'\oplus x_i'')$ repeat $m$ times.

We also have
\begin{eqnarray}
\label{Eq:eq1}
  &&(a-20\varepsilon)_+\nonumber\\
  &&\precsim (\beta_n\alpha(a)-3\varepsilon)_+\oplus (\beta_n'\alpha'\gamma_n(a)-4\varepsilon)_+\oplus (\gamma_n'\gamma_n(a)-3\varepsilon)_+\nonumber\\
  &&\precsim (\beta_n\alpha(a)-3\varepsilon)_+\oplus (\beta_n'\alpha'\gamma_n(a)-4\varepsilon)_+\oplus (\gamma_n'\gamma_n(1_A)-\varepsilon)_+\nonumber\\
    &&\precsim (\beta_n\alpha(a)-3\varepsilon)_+\oplus (\beta_n'\alpha'\gamma_n(a)-4\varepsilon)_+\oplus r\nonumber\\
  &&\precsim (\beta_n\alpha(a)-2\varepsilon)_+\oplus (\beta_n'\alpha'\gamma_n(a)-4\varepsilon)_+\nonumber\\
 &&\precsim \oplus_{i=1}^{n}(y_i'\oplus x_i'').\nonumber
\end{eqnarray}

\textbf{(1.2)} If  $(a'-3\varepsilon)_+$  is not Cuntz equivalent to a projection.

By Theorem 2.1 (4) in \cite{EFF},  there is a non-zero positive element $d$  such that
 $(a'-4\varepsilon)_++d\precsim (a'-3\varepsilon)_+$.

Since $ (\alpha(a)-3\varepsilon)_+ \precsim \oplus_{i=1}^{n}x_i'$, .
Since $kb_1\precsim \alpha(a)$, there exist $v\in M_k(B)_+$ such that
$$\|v^*diag(\alpha(a), 0\otimes 1_{k-1}v-b_1\otimes 1\|<\bar{\varepsilon}.$$
We assume that $\|v\|\leq M(\bar{\varepsilon})$, by $(4)$, there exists sufficiently large integer $n$ such that $$\|\beta_n(v^*)diag(\beta_n\alpha(a), 0\otimes 1_{k-1}\beta_n(v)-\beta_n(b_1)\otimes 1\|<\varepsilon'.$$
Therefore we have $$k(\beta_n(b_1)-2\varepsilon')_+\precsim (\beta_n\alpha(a)-\varepsilon')_+.$$

With  $F=\{\gamma_n(a)\},$   any $\varepsilon'>0$ with $\varepsilon'$ sufficiently small,  let $E=\gamma_n(1_A)A\gamma_n(1_A)$, by Lemma 2.7, $E$ is  asymptotically tracially in $\mathcal{P}$, there exist
a ${\rm C^*}$-algebra $D$  in $\mathcal{P}$ and completely positive  contractive linear maps  $\alpha':E\to D$ and  $\beta_n': D\to E$, and $\gamma_n':E\to E\cap\beta_n'(D)^{\perp}$ such that

$(1)'$ the map $\alpha'$ is unital  completely positive   linear map, $\beta_n'(1_D)$ and $\gamma_n'(1_E)$ are all projections, $\beta_n'(1_D)+\gamma_n'(1_E)=1_E$ for all $n\in \mathbb{N}$,

$(2)'$ $\|x-\gamma_n'(x)-\beta_n'(\alpha'(x))\|<\varepsilon''$ for all $x\in G$ and for all $n\in {\mathbb{N}}$,

$(3)'$ $\alpha'$ is an $G$-$\varepsilon''$-approximate embedding,

$(4)'$ $\lim_{n\to \infty}\|\beta_n'(xy)-\beta_n'(x)\beta_n(y)\|=0$ and $\lim_{n\to \infty}\|\beta_n'(x)\|=\|x\|$ for all $x,y\in D$, and

$(5)'$ $\gamma_n'(\gamma(1_A))\precsim \beta_n\alpha(c)$ for all $n\in \mathbb{N}$.

Since $D$ is weakly $(m,n)$-divisible,
 there exist $x_1'',x_2'', \cdots,x_n'' \in {\rm M}_{\infty}(D)_+$ such that
$x_j''\oplus x_j''\oplus \cdots \oplus x_j''\precsim  (\beta_n'\alpha'\gamma_n(a)-2\varepsilon)_+$ where $x_j''$ repeat $m$ times  and $(\beta_n'\alpha'\gamma_n(a)-3\varepsilon)_+\precsim \oplus_{i=1}^{n}x_i''$.
With the same argument as above we have
  $$k(\beta_n'(b_2)-2\varepsilon')_+\precsim (\beta_n'\alpha'\gamma_n(a)-\varepsilon')_+$$ and
  $$(\beta_n'\alpha'\gamma_n(a)-6\varepsilon')_+\precsim (\beta_n'((b_1-2\varepsilon')_+).$$

Therefore we have
\begin{eqnarray}
\label{Eq:eq1}
&&(x_j'\oplus x_j''))\oplus(x_j'\oplus x_j''))\oplus\cdots \oplus(x_j'\oplus x_j''))\nonumber\\
   &&\precsim  (a'-2\varepsilon)_+\oplus  (\beta_n'\alpha'\gamma_n(a)-2\varepsilon)_+\nonumber\\
&&\precsim a,\nonumber
\end{eqnarray}

for $1\leq j\leq n$  where $(x_j'\oplus x_j'')$ repeat $m$ times.

 We also have
\begin{eqnarray}
\label{Eq:eq1}
  &&(a-20\varepsilon)_+\nonumber\\
  &&\precsim (\beta_n\alpha(a)-4\varepsilon)_+\oplus (\beta_n'\alpha'\gamma_n(a)-4\varepsilon)_+\oplus (\gamma_n'\gamma_n(a)-2\varepsilon)_+\nonumber\\
  &&\precsim (\beta_n\alpha(a)-4\varepsilon)_+\oplus (\beta_n'\alpha'\gamma_n(a)-4\varepsilon)_+\oplus (\gamma_n'\gamma_n(1_A)-\varepsilon)_+\nonumber\\
    &&\precsim (\beta_n\alpha(a)-4\varepsilon)_+\oplus (\beta_n'\alpha'\gamma_n(a)-4\varepsilon)_+\oplus d\nonumber\\
  &&\precsim (\beta_n\alpha(a)-3\varepsilon)_+\oplus (\beta_n'\alpha'\gamma_n(a)-4\varepsilon)_+\nonumber\\
 &&\precsim \oplus_{i=1}^{n}(x_i'\oplus x_i'').\nonumber
\end{eqnarray}

\textbf{ Case (2)}, If  $(a'-2\varepsilon)_+$  is not Cuntz equivalent to a projection.

By Theorem 2.1 (4) in \cite{EFF},  there is a non-zero positive element $d$  such that
 $(a'-3\varepsilon)_++d\precsim (a'-2\varepsilon)_+$.

With  $F=\{\gamma_n(a)\},$   any $\varepsilon'>0$ with $\varepsilon'$ sufficiently small,  let $E=\gamma_n(1_A)A\gamma_n(1_A)$, by Lemma 2.7, $E$ is  asymptotically tracially in $\mathcal{P}$, there exist
a ${\rm C^*}$-algebra $D$  in $\mathcal{P}$ and completely positive  contractive linear maps  $\alpha':E\to D$ and  $\beta_n': D\to E$, and $\gamma_n':E\to E\cap\beta_n'(D)^{\perp}$ such that

$(1)'$ the map $\alpha'$ is unital  completely positive   linear map, $\beta_n'(1_D)$ and $\gamma_n'(1_E)$ are all projections, $\beta_n'(1_D)+\gamma_n'(1_E)=1_E$ for all $n\in \mathbb{N}$,

$(2)'$ $\|x-\gamma_n'(x)-\beta_n'(\alpha'(x))\|<\varepsilon''$ for all $x\in G$ and for all $n\in {\mathbb{N}}$,

$(3)'$ $\alpha'$ is an $G$-$\varepsilon''$-approximate embedding,

$(4)'$ $\lim_{n\to \infty}\|\beta_n'(xy)-\beta_n'(x)\beta_n(y)\|=0$ and $\lim_{n\to \infty}\|\beta_n'(x)\|=\|x\|$ for all $x,y\in D$, and

$(5)'$ $\gamma_n'(\gamma(1_A))\precsim \beta_n\alpha(c)$ for all $n\in \mathbb{N}$.

Since $D$ is weakly $(m,n)$-divisible,
 there exist $x_1'',x_2'', \cdots,x_n'' \in {\rm M}_{\infty}(D)_+$ such that
$x_j''\oplus x_j''\oplus \cdots \oplus x_j''\precsim  (\beta_n'\alpha'\gamma_n(a)-2\varepsilon)_+$ where $x_j''$ repeat $m$ times  and $(\beta_n'\alpha'\gamma_n(a)-3\varepsilon)_+\precsim \oplus_{i=1}^{n}x_i''$.
With the same argument as above we have
  $$k(\beta_n'(b_2)-2\varepsilon')_+\precsim (\beta_n'\alpha'\gamma_n(a)-\varepsilon')_+$$ and
  $$(\beta_n'\alpha'\gamma_n(a)-6\varepsilon')_+\precsim (\beta_n'((b_1-2\varepsilon')_+).$$

 Therefore we have

\begin{eqnarray}
\label{Eq:eq1}
  &&(y_i'\oplus x_i'')\oplus(y_i'\oplus x_i'')\oplus\cdots \oplus(y_i'\oplus x_i'')\nonumber\\
 &&\precsim  (\beta_n\alpha(a)-2\varepsilon)_+\oplus  (\beta_n'\alpha'\gamma_n(a)-2\varepsilon)_+\nonumber\\
&&\precsim a,\nonumber
\end{eqnarray}
for  $1\leq i\leq n$ where $(x_i'\oplus x_i'')$ repeat $m$ times.

We also have
\begin{eqnarray}
\label{Eq:eq1}
  &&(a-20\varepsilon)_+\nonumber\\
  &&\precsim (\beta_n\alpha(a)-3\varepsilon)_+\oplus (\beta_n'\alpha'\gamma_n(a)-4\varepsilon)_+\oplus (\gamma_n'\gamma_n(a)-2\varepsilon)_+\nonumber\\
  &&\precsim (\beta_n\alpha(a)-3\varepsilon)_+\oplus (\beta_n'\alpha'\gamma_n(a)-4\varepsilon)_+\oplus (\gamma_n'\gamma_n(1_A)-\varepsilon)_+\nonumber\\
    &&\precsim (\beta_n\alpha(a)-3\varepsilon)_+\oplus (\beta_n'\alpha'\gamma_n(a)-4\varepsilon)_+\oplus d\nonumber\\
   &&\precsim (\beta_n\alpha(a)-2\varepsilon)_+\oplus (\beta_n'\alpha'\gamma_n(a)-4\varepsilon)_+\nonumber\\
 &&\precsim \oplus_{i=1}^{n}(y_i'\oplus x_i'').\nonumber
\end{eqnarray}
\end{proof}

 \end{document}